\documentclass[12pt]{amsart}

\usepackage[all,2cell,ps]{xy}
\usepackage{amssymb,amstext,amsmath,amscd,amsthm,amsfonts,enumerate,graphicx,latexsym,color}

\newtheorem{thm}{Theorem}[section]
\newtheorem{lem}[thm]{Lemma}
\newtheorem{prop}[thm]{Proposition}
\newtheorem{cor}[thm]{Corollary}
\newtheorem{ques}[thm]{Question}
\newtheorem{conj}[thm]{Conjecture}

\theoremstyle{definition}

\newtheorem{eg}[thm]{Example}

\newtheorem{statement}[thm]{}

\theoremstyle{remark}
\newtheorem{rmk}[thm]{Remark}

\newcommand{\ZZ}{\mathbb{Z}}

\newcommand{\QQ}{\mathbb{Q}}

\newcommand{\tensor}{\otimes}

\DeclareMathOperator{\Tor}{Tor}
\DeclareMathOperator{\Ext}{Ext}
\DeclareMathOperator{\Hom}{Hom}

 \DeclareMathOperator{\Spec}{Spec}

 \DeclareMathOperator{\pd}{pd}

 \DeclareMathOperator{\depth}{depth}

 \DeclareMathOperator{\htt}{height}
\def\p{\mathfrak{p}}

\newcommand{\E}{\textup{E}}

\def\urltilda{\kern -.15em\lower .7ex\hbox{\~{}}\kern .04em}\def\urldot{\kern -.10em.\kern -.10em}\def\urlhttp{http\kern -.10em\lower -.1ex\hbox{:}\kern -.12em\lower 0ex\hbox{/}\kern -.18em\lower 0ex\hbox{/}}

\def\Tr{\mathrm{Tr}}
\DeclareMathOperator{\RHom}{{\bf R}Hom}
\def\cext{\mathrm{\widehat{Ext}}\mathrm{}}
\def\lhom{\operatorname{\underline{Hom}}}

\def\ac{\mathsf{(AC)}}
\def\arc{\mathsf{(ARC)}}
\def\garcc{\mathsf{(GARC')}}
\def\garc{\mathsf{(GARC)}}
\def\sac{\mathsf{(SAC)}}
\def\sacc{\mathsf{(SACC)}}
\def\fed{\mathsf{(FED)}}
\def\hwc{\mathsf{(HWC)}}

\numberwithin{equation}{section}

\begin{document}
\baselineskip=15pt
\title{Auslander-Reiten conjecture and Auslander-Reiten duality}
\author{Olgur Celikbas}
\address{Department of Mathematics, 323 Mathematical Sciences Bldg, University of Missouri--Columbia, Columbia, MO 65211 USA}
\email{celikbaso@missouri.edu}
\urladdr{http://www.math.missouri.edu/~celikbaso/}
\author{Ryo Takahashi}
\address{Department of Mathematical Sciences, Faculty of Science, Shinshu University, 3-1-1 Asahi, Matsumoto, Nagano 390-8621, Japan/Department of Mathematics, University of Nebraska, Lincoln, NE 68588-0130, USA}
\curraddr{Graduate School of Mathematics, Nagoya University, Furocho, Chikusaku, Nagoya 464-8602, Japan}
\email{takahashi@math.nagoya-u.ac.jp}
\urladdr{http://www.math.nagoya-u.ac.jp/~takahashi/}
\date{\today}
\thanks{The second author was partially supported by JSPS Grant-in-Aid for Young Scientists (B) 22740008 and by JSPS Postdoctoral Fellowships for Research Abroad}
\begin{abstract}
Motivated by a result of Araya, we extend the Auslander-Reiten duality theorem to Cohen-Macaulay local rings.
We also study the Auslander-Reiten conjecture, which is rooted in Nakayama's work on finite dimensional algebras.
One of our results detects a certain condition that forces the conjecture to hold over local rings of positive depth.
\end{abstract}
\maketitle

\section{Introduction}

The \emph{Auslander-Reiten conjecture} \cite{AR} claims that, over an Artin algebra $R$, if $M$ is a finitely generated $R$-module such that $\Ext^{>0}_{R}(M, M\oplus R)=0$, then $M$ is projective.
This long-standing conjecture can be stated over any Noetherian ring $R$ and is known to be true over several classes of algebras.
For example, Auslander and Reiten \cite{AR} proved the conjecture for algebras of finite representation type.
Hoshino \cite{Hos} determined another noteworthy class; it holds for symmetric Artin algebras with radical cube zero.
The Auslander-Reiten conjecture is closely related to other important conjectures such as the \emph{Tachikawa conjecture} \cite{T} and the \emph{finitistic dimension conjecture} \cite{Happel}.
Indeed it is rooted in a conjecture of Nakayama \cite{Nak} and the one known as the \emph{generalized Nakayama conjecture} \cite{AR}: If $R$ is an Artin algebra, then every indecomposable injective $R$-module occurs as a direct summand of one of the terms of the minimal injective resolution of $R$.
We refer the reader to \cite{ABS,Buchweitz2,CH1,Di,DP,Wei2,Wei1,Yamagata} for more information on the relation of these homological conjectures.

Christensen and Holm \cite{CH1} proved that the Auslander-Reiten conjecture holds over rings satisfying the Auslander condition.
More precisely, they proved that the implication
$$
\xymatrix{
\ac \ar@{=>}[r] & \arc
}
$$
holds for a (left) Noetherian ring.
Here, $\ac$ and $\arc$ denote the {\em Auslander condition} and the {\em Auslander-Reiten condition}, which are conditions defined for a fixed Noetherian ring $R$, as follows.
\begin{quote}
\begin{enumerate}
\item[$\ac$]
For every finitely generated $R$-module $M$ there exists a nonnegative integer $b_{M}$ such that for every finitely generated $R$-module $N$ one has: if $\Ext^{\gg 0}_{R}(M,N)=0$, then $\Ext^{>b_{M}}_{R}(M,N)=0$.
\item[$\arc$]
For every finitely generated $R$-module $M$, if $\Ext^{>0}_{R}(M,M\oplus R)=0$, then $M$ is projective.
\end{enumerate}
\end{quote}
Although there exist rings that do not satisfy $\ac$ \cite{JoS}, this result gave new insight on the Auslander-Reiten conjecture and other related homological conjectures.

Commutative Noetherian local rings known to satisfy $\arc$ include Gorenstein local rings of codimension at most four \cite{S}.
Recently Araya \cite{A} proved that if all Gorenstein local rings of dimension at most one satisfy $\arc$ so do all Gorenstein local rings.
The main tool he used for the proof of his result is a duality theorem for Gorenstein local rings, which is known as the {\em Auslander-Reiten duality} theorem (cf. \cite{Au2,Buchweitz,Y}).
On the other hand, as a natural generalization of $\ac$, the following {\em generalized Auslander-Reiten condition} has also been investigated by several authors \cite{Di,Wei2,Wei1}.
\begin{quote}
\begin{enumerate}
\item[$\garc$]
For every finitely generated $R$-module $M$, if there exists a nonnegative integer $n$ such that $\Ext^{>n}_{R}(M,M\oplus R)=0$, then $M$ has projective dimension at most $n$.
\end{enumerate}
\end{quote}

In this paper, we consider refining the Auslander-Reiten duality theorem and the Christensen-Holm theorem stated above.

First, we shall prove the following theorem, which extends the Auslander-Reiten duality theorem to arbitrary Cohen-Macaulay local rings and deduces Araya's main observation from a more general context.
\begin{thm}
Let $(R,\mathfrak{m},k)$ be a $d$-dimensional Cohen-Macaulay local ring with a canonical module $\omega$. 
Let $X$ be a totally reflexive $R$-module, and let $M$ be a maximal Cohen-Macaulay $R$-module that is locally free on the punctured spectrum of $R$.
Then, for all integers $i$ one has an isomorphism
$$
\cext_R^i(X,M)\cong\cext_R^{(d-1)-i}(X^\ast,M^\dag)^\vee,
$$
where $(-)^\ast=\Hom_R(-,R)$, $(-)^\dag=\Hom_R(-,\omega)$, $(-)^\vee=\Hom_R(-,\mathrm{E}_R(k))$ and $\cext$ denotes Tate cohomology.
\end{thm}

Next, we define the following conditions over a commutative Noetherian local ring $R$.
\begin{quote}
\begin{enumerate}
\item[$\sac$]
For every finitely generated $R$-module $M$, if $\Ext^{>0}_{R}(M,R)=0$ and $\Ext^{\gg 0}_{R}(M,M)=0$, then $\Ext^{>0}_{R}(M,M)=0$.
\item[$\sacc$]
For every finitely generated $R$-module $M$ that has constant rank, if $\Ext^{>0}_{R}(M,R)=0$ and $\Ext^{\gg 0}_{R}(M,M)=0$, then $\Ext^{>0}_{R}(M,M)=0$.
\end{enumerate}
\end{quote}
We shall prove the following, which supplements the implication shown by Christensen and Holm with the conditions introduced above.
\begin{thm}
For each commutative Noetherian local ring $R$, the implications
$$
\xymatrix{
\ac \ar@{=>}[r] & \sac \ar@{=>}[r] & \sacc \ar@{=>}[r]_{(\ast)} & \arc \\
& \garc \ar@{<=>}[u]
}
$$
hold, where $R$ is assumed to have positive depth for $(\ast)$.
\end{thm}
\noindent
Here, the assumption in $(\ast)$ that $R$ has positive depth is reasonable, because $R$ satisfies $\sacc$ trivially if $\depth R=0$.

We also give an application of Araya's result and pose a question concerning the Auslander-Reiten conjecture for modules having bounded Betti numbers.
We finish the article by discussing in Section 5 the relation of the Auslander-Reiten conjecture with a question of Huneke and Wiegand on the tensor products of modules.

\section{Basic definitions}

This section is devoted to giving basic definitions and their basic properties which will often be used later.

\begin{statement}
Throughout this paper, $R$ denotes a left Noetherian ring, and modules considered over $R$ are finitely generated left modules.
Unless otherwise specified, $R$ is assumed to be a commutative local ring with unique maximal ideal $\mathfrak{m}$ and residue field $k=R/\mathfrak m$.
\end{statement}

\begin{statement}
For a given $R$-module $M$, we set
$$
M^\ast=\Hom_R(M,R),\ M^{\vee}=\Hom_{R}(M, \E_{R}(k)),\ M^{\dag}=\Hom_{R}(M, \omega_{R}),
$$
where $\E_{R}(k)$ is the injective hull of $k$, and $\omega_{R}$ is a canonical module (in the case where $R$ is a Cohen-Macaulay ring that is a homomorphic image of a Gorenstein local ring).
Clearly, $M^{\dag}=M^{\ast}$ when $R$ is Gorenstein.
\end{statement}

\begin{statement}
We denote by $U_R$ the {\em punctured spectrum} of $R$, i.e., $U_{R}=\Spec R\setminus\{ \mathfrak{m}\}$.
We say that $M$ is {\em locally free on $U_R$} provided that $M_{p}$ is a free $R_{p}$-module for all $p\in U_{R}$.
The ring $R$ is called an {\em isolated singularity} if $R_{p}$ is a regular local ring for all $p\in U_{R}$.
\end{statement}

\begin{statement}
We use the notation $\overset{\rm st}{\cong}$ to denote a \emph{stable isomorphism}, that is, for two $R$-modules $M,N$ we write $M\overset{\rm st}{\cong}N$ if $M\oplus F\cong N\oplus G$ for some free $R$-modules $F,G$.
\end{statement}

\begin{statement}\label{2.5}
Let $M$ be an $R$-module.
For an integer $i>0$, we denote by $\Omega^{i}M$ the $i$-th {\em syzygy} of $M$, namely, the image of the $i$-th differential map in a minimal free resolution of $M$.
As a convention, we set $\Omega^0M=M$.
For an $R$-module $M$, we denote by $\Tr M$ the {\em (Auslander) transpose} of $M$, which is defined as the cokernel of the $R$-dual map $\delta_1^\ast$ of the first differential map $\delta_1$ in a minimal free resolution of $M$.
Here are some remarks:
\begin{enumerate}[(1)]
\item
The modules $\Omega^iM$ and $\Tr M$ are uniquely determined up to isomorphism, since so is a minimal free resolution of $M$.
\item
There is a stable isomorphism $\Omega^2\Tr M\overset{\rm st}{\cong}M^\ast$.
\end{enumerate}
\end{statement}

\begin{statement}
An $R$-module $X$ is called \emph{totally reflexive} if the natural homomorphism $X\to X^{\ast\ast}$ is an isomorphism and $\Ext^i_{R}(X,R)=\Ext_R^i(X^{\ast},R)=0$ for all $i>0$.
Here are several properties of totally reflexive modules which we will use.
For the details, we refer to \cite{AuBr} and \cite{Ch}.
\begin{enumerate}[(1)]
\item
A totally reflexive module over a Cohen-Macaulay local ring is maximal Cohen-Macaulay.
Over a Gorenstein local ring, the totally reflexive modules are precisely the maximal Cohen-Macaulay modules.
\item
For a totally reflexive $R$-module $X$ and an integer $i>0$, we define the $i$-th {\em cosyzygy} of $X$, denoted by $\Omega^{-i}X$, to be the image of the $R$-dual map $\delta_i^\ast$ of the $i$-th differential map in a minimal free resolution of $X^\ast$.
We have $\Omega^i\Omega^jX\overset{\rm st}{\cong}\Omega^{i+j}X$ for all integers $i,j$.
\item
If $X$ is a totally reflexive $R$-module, then so are $X^\ast$, $\Tr X$ and $\Omega^iX$ for all integers $i$.
\end{enumerate}
\end{statement}

\begin{statement}
For $R$-modules $M,N$, we denote by $\lhom_R(M,N)$ the residue $R$-module of $\Hom_R(M,N)$ by the $R$-submodule consisting of the homomorphisms from $M$ to $N$ that factor through free modules.
Here are some remarks:
\begin{enumerate}[(1)]
\item
We have $\lhom_R(M,N)\cong\Tor_1^R(\Tr M,N)$; see \cite[(3.9)]{Y}.
\item
If $M$ or $N$ is locally free on $U_R$, then the $R$-module $\lhom_R(M,N)$ has finite length.
\end{enumerate}
\end{statement}

\begin{statement}\label{2.8}
Let $M$ be an $R$-module, and let $X$ be a totally reflexive $R$-module.
We denote by $\cext_R^i(X,M)$ the $i$-th {\em Tate cohomology} module, which is defined as
$$
\cext_R^i(X,M)=\lhom_R(\Omega^iX,M).
$$
We will use the following properties of Tate cohomology.
We can find details of Tate cohomology, for instance, in \cite[\S 7]{catgp} (the paper \cite{catgp} assumes that the base local ring is Henselian, but this assumption is unnecessary in \cite[\S 7]{catgp}).
\begin{enumerate}[(1)]
\item
One has $\cext_R^i(X,M)\cong\Ext_R^i(X,M)$ for all positive integers $i$.
\item
One has $\cext_R^i(\Omega^jX,M)\cong\cext_R^{i+j}(X,M)$ for all integers $i,j$.
\item
If $M$ is also totally reflexive, then $\cext_R^i(X,\Omega^jM)\cong\cext_R^{i-j}(X,M)$ for all integers $i,j$.
\end{enumerate}
\end{statement}

\section{A generalization of Auslander-Reiten duality}

Araya \cite[Theorem 8]{A} made use of the Auslander-Reiten duality theorem over Gorenstein local rings \cite{Y} and proved the following theorem:

\begin{thm}[Araya]\label{Araya}
Let $R$ be a $d$-dimensional Gorenstein local ring and let $M$ be a maximal Cohen-Macaulay $R$-module that is locally free on $U_R$. If $\cext_{R}^{d-1}(M,M)=0$, then $M$ is free.
\end{thm}

It is a consequence of a result of Araya \cite{A} that if all one dimensional Gorenstein rings satisfy the Auslander-Reiten conjecture discussed in the introduction, so do all Gorenstein local rings.
Our aim in this section is to deduce Theorem \ref{Araya} from a more general result, Theorem \ref{d}, that holds over arbitrary Cohen-Maculay local rings.
Such a generalization could be useful for researchers who wish to further study the Auslander-Reiten conjecture. It also seems interesting in itself, as it concerns the vanishing of cohomology over Cohen-Macaulay local rings. We start with a few preliminary results:

\begin{lem}\label{a}
\begin{enumerate}[\rm(1)]
\item
Let $R$ be a Cohen-Macaulay local ring with a canonical module $\omega$ and let $M$ and $N$ be maximal Cohen-Macaulay $R$-modules.
Then $\Ext_R^i(M,N^\dag)\cong\Ext_R^i(N,M^\dag)$ for all $i\in\ZZ$.
\item
Let $R$ be a local ring and let $X$ and $Y$ be totally reflexive $R$-modules.
Then $\Ext_R^i(X^\ast,Y^\ast)\cong\Ext_R^i(Y,X)$ for all $i\in\ZZ$.
\end{enumerate}
\end{lem}

\begin{proof}
(1) The following isomorphisms hold in the derived category of $R$:
\begin{align*}
\RHom_R(M,N^\dag) & \cong \RHom_R(M,\RHom_R(N,\omega))
\cong \RHom_R(M\tensor_R^{\bf L}N,\omega)\\
& \cong \RHom_R(N,\RHom_R(M,\omega))
\cong \RHom_R(N,M^\dag).
\end{align*}
Here, the second and third isomorphisms follow from \cite[(A.4.21)]{Ch}.
One deduces the required isomorphism by taking the $i$-th cohomology.

(2) The conclusion follows from the following isomorphisms:
\begin{align*}
\RHom_R(X^\ast,Y^\ast) & \cong \RHom_R(X^\ast,\RHom_R(Y,R))
\cong \RHom_R(X^\ast\tensor_R^{\bf L}Y,R) \\
& \cong \RHom_R(Y,\RHom_R(X^\ast,R))
\cong \RHom_R(Y,X).
\end{align*}
\end{proof}

\begin{lem}\label{b}
Let $R$ be a local ring and let $X$ be a totally reflexive $R$-module.
Then $\Omega^{-i}(X^\ast)\overset{\rm st}{\cong}(\Omega^iX)^\ast$ for all $i\in\ZZ$.
\end{lem}

\begin{proof}
Let $i\ge0$.
Taking a minimal free resolution $(\cdots \to F_1 \to F_0 \to 0)$ of $X$, we have an exact sequence $0 \to \Omega^iX \to F_{i-1} \to \cdots \to F_0 \to X \to 0$.
Since $X$ is totally reflexive, the $R$-dual sequence
$$
0 \to X^\ast \to F_0^\ast \to \cdots \to F_{i-1}^\ast \to (\Omega^iX)^{\ast} \to 0
$$
is also exact.
Hence $X^\ast\overset{\rm st}{\cong}\Omega^i((\Omega^iX)^\ast)$.
Applying $\Omega^{-i}$ to this, we get $\Omega^{-i}(X^\ast)\overset{\rm st}{\cong}(\Omega^iX)^\ast$.

Now let $i\le0$.
Set $j=-i$ and $Y=X^\ast$.
Then, since $j\ge0$ and $Y$ is totally reflexive, we have $\Omega^{-j}(Y^\ast)\overset{\rm st}{\cong}(\Omega^jY)^\ast$.
Hence $\Omega^iX\overset{\rm st}{\cong}(\Omega^{-i}(X^\ast))^\ast$.
Applying $(-)^\ast$, we get $(\Omega^iX)^\ast\overset{\rm st}{\cong}\Omega^{-i}(X^\ast)$.
\end{proof}

We are now ready to prove the main result of this section:

\begin{thm}\label{d}
Let $R$ be a $d$-dimensional Cohen-Macaulay local ring with a canonical module $\omega$. 
Let $X$ be a totally reflexive $R$-module.
Let $M$ be a maximal Cohen-Macaulay $R$-module that is locally free on $U_{R}$.
Then, for all $i\in \ZZ$,
$$
\cext_R^i(X,M)\cong\cext_R^{(d-1)-i}(X^\ast,M^\dag)^\vee.
$$
\end{thm}

\begin{proof}
We have the following isomorphisms:
\begin{align*} 
\cext_R^i(X,M) 
& = \lhom_R(\Omega^iX,M)
= \mathrm{H}_{\mathfrak m}^0(\lhom_R(\Omega^iX,M))\\
& \cong \Ext_R^d(\lhom_R(\Omega^iX,M),\omega)^\vee
\cong \Ext_R^1(M,(\Omega^d\Tr\Omega^iX)^\dag)^\vee \\
& \cong \Ext_R^1(\Omega^d\Tr\Omega^iX,M^\dag)^\vee
\cong \cext_R^1(\Omega^d\Tr\Omega^iX,M^\dag)^\vee \\
& \cong \cext_R^{d-1}(\Omega^2\Tr\Omega^iX,M^\dag)^\vee
\cong \cext_R^{d-1}((\Omega^iX)^\ast,M^\dag)^\vee\\
& \cong \cext_R^{(d-1)-i}(\Omega^i((\Omega^iX)^\ast),M^\dag)^\vee
\cong \cext_R^{(d-1)-i}(X^\ast,M^\dag)^\vee.
\end{align*}
Here, since $M$ is locally free on $U_R$, the $R$-module $\underline{\Hom}_R(\Omega^iX,M)$ has finite length, which shows the second equality.
The first and second isomorphisms are obtained by \cite[(3.5.9)]{BH} and \cite[(3.10)]{Y}, respectively.
Lemmas \ref{a}(1) and \ref{b} give the third and last isomorphisms, respectively.
The other isomorphisms follow from \ref{2.5}(2) and \ref{2.8}.
\end{proof}

The following result is the original Auslander-Reiten duality theorem, which now follows from Theorem \ref{d}.

\begin{cor}[Auslander-Reiten duality]\label{ARD}
Let $R$ be a $d$-dimensional Gorenstein local ring and let $X$ and $M$ be maximal Cohen-Macaulay $R$-modules such that $M$ is locally free on $U_R$.
Then $\cext_R^i(X,M)\cong\cext_R^{(d-1)-i}(M,X)^\vee$ for all $i\in\ZZ$.
\end{cor}

\begin{proof}
In general, let $R$ be a local ring and $X,Y$ totally reflexive $R$-modules.
For an integer $i$ there are isomorphisms
\begin{align*}
\cext_R^i(X^\ast,Y^\ast)
& \cong \Ext_R^1(\Omega^{i-1}(X^\ast),Y^\ast)
\cong \Ext_R^1((\Omega^{1-i}X)^\ast,Y^\ast) \\
& \cong \Ext_R^1(Y,\Omega^{1-i}X)
\cong \cext_R^i(Y,X),
\end{align*}
where the second and third isomorphisms follow from Lemmas \ref{b} and \ref{a}(2), respectively.
Now the corollary follows from Theorem \ref{d}.
\end{proof}

\section{On a conjecture of Auslander and Reiten}

In this section we are concerned with the Auslander-Reiten conjecture (cf. the introduction,  \cite{ADS} or \cite{AR}) for commutative local rings.
Our purpose is to prove that an affirmative answer to a question of Christensen and Holm \cite{CH1} will show that all local rings that have positive depth satisfy the Auslander-Reiten conjecture, cf. Corollary \ref{corARC}.

Before proving the main result of this section stated as Theorem \ref{thmARC}, we will first discuss an application of Theorem \ref{Araya}; a weaker version of the Auslander-Reiten conjecture holds for certain periodic modules over even dimensional Gorenstein isolated singularities.

\begin{prop}\label{APP}
Let $R$ be a $d$-dimensional Gorenstein local ring, and let $M$ be a maximal Cohen-Macaulay $R$-module that is locally free on $U_R$.
Assume that $d$ is a positive even integer and that $M \overset{\rm st}{\cong} \Omega^{2n}M$ for some positive integer $n$.
If $\Ext^{i}_{R}(M,M)=0$ for all $i=1,3,5,\ldots, 2n-1$, then $M$ is free.
\end{prop}

\begin{proof}
First, note by assumption that $\Ext_R^i(M,M)=0$ for all odd integers $i$.
We set $N=\bigoplus_{0\le i\le n-1}\Omega^{2i}M$.
Then $N\overset{\rm st}{\cong} \Omega^{2}N$, and we have
\begin{align*}
\Ext^{1}_{R}(N,N)
& \cong \textstyle\bigoplus_{0\leq i, j \leq n-1}\Ext^{1}_{R}(\Omega^{2i}M, \Omega^{2j}M)
\cong \textstyle\bigoplus_{0\leq i, j \leq n-1}\Ext^{1}_{R}(\Omega^{2i+2n}M, \Omega^{2j}M)\\
& \cong \textstyle\bigoplus_{0\leq i, j \leq n-1}\cext^{1+2i+2n-2j}_{R}(M,M)
\cong \textstyle\bigoplus_{0\leq i, j \leq n-1}\Ext^{2(i+n-j+1)-1}_{R}(M,M)=0.
\end{align*}
Here, the fourth isomorphism follows from the fact that $1+2i+2n-2j=2(i+n-j+1)-1>0$.
Since $d$ is positive and even, we have $\cext_R^{d-1}(N,N)\cong\Ext_R^{d-1}(N,N)\cong\Ext_R^1(\Omega^{d-2}N,N)\cong\Ext_R^1(N,N)=0$.
Theorem \ref{Araya} implies that $N$, and hence $M$, is free.
\end{proof}

It follows by definition that a module $M$ with $M\overset{\rm st}{\cong}  \Omega^{n}M$ for some $n>0$ has bounded Betti sequence.
Therefore Theorem \ref{Araya} and Proposition \ref{APP} raise the following question:

\begin{ques}\label{qI}
Let $R$ be a $d$-dimensional Gorenstein local ring (that is not a complete intersection) with $d\geq 3$.
Let $M$ be a maximal Cohen-Macaulay $R$-module that is locally free on $U_R$.
Assume that $M$ has bounded Betti sequence.
If $\Ext^{i}_{R}(M,M)=0$ for all $i=1,2, \ldots, d-2$, then is $M$ free?
\end{ques}

There is a negative answer to Question \ref{qI} in case the ring considered is a complete intersection:

\begin{eg}
Let $R=k[[X,Y,Z,U]]/(XZ-YU)$ and $I=(X,Y)R$.
Then $R$ is a three dimensional hypersurface with an isolated singularity, and $I$ is a maximal Cohen-Macaulay $R$-module.
One can easily see that $\Ext^1_{R}(I,I)=0$ but $I$ is not free.
\end{eg}

The Auslander-Reiten conjecture (for commutative local rings) asserts that all commutative local rings satisfy the \emph{Auslander-Reiten condition} $\arc$, which is defined as follows:
\begin{quote}
\begin{enumerate}
\item[$\arc$]
For every $R$-module $M$, if $\Ext^{>0}_{R}(M,M\oplus R)=0$, then $M$ is free.
\end{enumerate}
\end{quote}
Some of the well-known examples of local rings that satisfy $\arc$ are complete intersection rings \cite{ADS}, Golod rings \cite{JoS} and Gorenstein rings of codimension at most four \cite{S}.
As discussed in the previous section, a recent result of Araya \cite{A} shows that if all Gorenstein local rings of dimension at most one satisfy the Auslander-Reiten conjecture, so do all Gorenstein local rings.
As complete intersections satisfy $\arc$, Araya's result in particular reproves a result of Huneke and Leuschke \cite{HL}: Gorenstein rings that are complete intersections in codimension one satisfy $\arc$.
Indeed, over local domains, this result is motivated by a question of Huneke and Wiegand \cite{HW} which we will discuss briefly at the end of this article (cf. also \cite{Ce}).
We refer the interested reader to \cite{CH1} for a survey of the Auslander-Reiten and related conjectures.

A local ring $R$ is said to satisfy the \emph{Auslander condition} $\ac$ on the vanishing of cohomology if the following condition holds:
\begin{quote}
\begin{enumerate}
\item[$\ac$]
For every $R$-module $M$ there exists a nonnegative integer $b_{M}$ such that for every $R$-module $N$ one has: if $\Ext^{\gg 0}_{R}(M,N)=0$, then $\Ext^{>b_{M}}_{R}(M,N)=0$.
\end{enumerate}
\end{quote}
Christensen and Holm proved in \cite{CH1} that $\ac$ implies $\arc$.
More precisely, the following result follows from \cite[(2.3)]{CH1}:
\begin{thm}[Christensen-Holm]\label{thmCH}
Let $R$ be a local ring that satisfies $\ac$, and let $M$ be an $R$-module.
If $\Ext^{>0}_{R}(M,R)=0$ and $\Ext^{\gg 0}_{R}(M,M)=0$, then $M$ is free.
\end{thm}

As the vanishing conditions imposed on the module $M$ considered in Theorem \ref{thmCH} appear to be weaker than those in the Auslander-Reiten Conjecture, Christensen and Holm asked in \cite[2.4]{CH1} whether the following two conditions are equivalent.
\begin{enumerate}
\item $\Ext^{>0}_{R}(M,R)=0$ and $\Ext^{>0}_{R}(M,M)=0$.
\item $\Ext^{>0}_{R}(M,R)=0$ and $\Ext^{\gg 0}_{R}(M,M)=0$.
\end{enumerate}
An example of Schulz \cite{Sch} yields a self-injective noncommutative Artin algebra over which the conditions (1) and (2) are \emph{not} equivalent.
As stated in \cite{CH1}, for commutative rings, it is not known whether $(2)$ implies $(1)$, in fact not even for Gorenstein rings.
Notice that Theorem \ref{thmCH} shows that if the ring satisfies $\ac$, then $(1)$ and $(2)$ are equivalent (cf. also \cite{Di}).
This motivates us to consider the following two conditions.
\begin{quote}
\begin{enumerate}
\item[$\sac$]
For every $R$-module $M$, if $\Ext^{>0}_{R}(M,R)=0$ and $\Ext^{\gg 0}_{R}(M,M)=0$, then $\Ext^{>0}_{R}(M,M)=0$.
\item[$\sacc$]
For every $R$-module $M$ that has constant rank, if $\Ext^{>0}_{R}(M,R)=0$ and $\Ext^{\gg 0}_{R}(M,M)=0$, then $\Ext^{>0}_{R}(M,M)=0$.
\end{enumerate}
\end{quote}
Here $\sac$ and $\sacc$ stand for the \emph{symmetric Auslander condition} and the \emph{symmetric Auslander condition for modules with constant rank}.
Recall that an $R$-module $M$ is said to be of \emph{constant rank} if there exists an integer $r$ such that $M_{p}\cong R^{(r)}_{p}$ for all associated prime ideals $p$ of $R$.

Now we investigate the conditions $\arc,\sac,\sacc$ under modding out by a nonzerodivisor.

\begin{thm}\label{thmARC}
Let $R$ be a local ring and let $x\in R$ be a nonzerodivisor on $R$.
\begin{enumerate}[\rm(1)]
\item
If $R/xR$ satisfies $\sac$/$\arc$, then $R$ satisfies $\sac$/$\arc$.
\item
If $R$ satisfies $\sacc$, then $R/xR$ satisfies $\arc$.
\end{enumerate} 
\end{thm}

\begin{proof} We set $S=R/xR$.

(1) Let $M$ be an $R$-module with $\Ext_{R}^{>0}(M,R)=0$.
Then we have $\Ext_{R}^{>0}(\Omega_{R} M,R)=0$, and $\Ext_R^i(\Omega_RM,\Omega_RM)\cong\Ext_R^{i+1}(M,\Omega_RM)\cong\Ext_R^i(M,M)$ for $i>0$.
Therefore, to prove that $R$ satisfies $\sac$/$\arc$, we may replace $M$ with $\Omega_{R}M$ and hence assume that $x$ is a nonzerodivisor on both $R$ and $M$.
(Note that, if  $\Ext^{1}_{R}(M,M)=0$, then $\Omega_{R}M$ is free if and only if $M$ is free \cite[Lemma 1(iii), page 154]{Mat}.)

We consider the short exact sequence $0 \to M  \stackrel{x}{\rightarrow} M \to M/xM \to 0$.
Let $n\ge0$.
Nakayama's lemma implies that $\Ext_{R}^{>n}(M,M)=0$ if and only if $\Ext^{>n}_{R}(M,M/xM)=0$.
Since $\Ext^{i}_{R}(M,M/xM)\cong \Ext^{i}_{S}(M/xM,M/xM)$ for $i\ge0$ by \cite[Lemma 2(ii), page 140]{Mat}, we deduce that $\Ext_{R}^{>n}(M,M)=0$ if and only if $\Ext^{>n}_{S}(M/xM,M/xM)=0$.
Similarly, $\Ext^{>0}_{S}(M/xM, S)=0$ since $\Ext^{>0}_R(M,R)=0$. 
Therefore, if $S$ satisfies $\sac$, then so does $R$.
Furthermore, if $\Ext_R^{>0}(M,M)=0$ and $S$ satisfies $\arc$, then $M/xM$ is free over $S$ and hence $M$ is free over $R$.
Thus, $R$ satisfies $\arc$.

(2) Let $M$ be an $S$-module such that $\Ext_{S}^{>0}(M,M\oplus S)=0$.
We would like to prove that $M$ is a free $S$-module.
Using the snake lemma, we see that the multiplication by $x$ on the natural exact sequence $0 \to \Omega_{R}M \to R^{(n)} \to M \to 0$ yields an exact sequence $0 \to M \to \Omega_{R}M/x(\Omega_{R}M) \to S^{(n)} \to M \to 0$.
Hence we have a short exact sequence
$$
0 \to M \to \Omega_{R}M/x(\Omega_{R}M) \to \Omega_{S}M \to 0.
$$
This, being an element of $\Ext^{1}_{S}(\Omega_{S}M,M)\cong\Ext^{2}_{S}(M,M)=0$, splits and hence yields $\Omega_{R}M/x(\Omega_{R}M) \cong M \oplus \Omega_{S}M$.
For all positive integers $i$, we have:
\begin{align*}
\Ext^{i}_{S}(M,\Omega_{S}M)
& \cong \Ext^{i}_{S}(M, M \oplus \Omega_{S}M)
\cong \Ext^{i}_{S}(M, \Omega_{R}M/x(\Omega_{R}M))\\
& \cong  \Ext^{i+1}_{R}(M, \Omega_{R}M)
\cong   \Ext^{i}_{R}(\Omega_{R}M, \Omega_{R}M),
\end{align*}
where the third isomorphism holds by \cite[Lemma 2, page 140]{Mat}.
Since $\Ext^{i}_{S}(M,M) \cong \Ext^{i+1}_{S}(M,\Omega_{S}M)$ for all $i>0$, we have that $\Ext^{>1}_{R}(\Omega_{R}M, \Omega_{R}M)=0$.
Note that $\Omega_{R}M$ is of constant rank and $\Ext^{>0}_{R}(\Omega_{R}M, R)=0$. Therefore, since $R$ satisfies $\sacc$, $\Ext^{>0}_{R}(\Omega_{R}M, \Omega_{R}M)=0$.
Now $\Ext^{1}_{S}(M, \Omega_{S}M)\cong\Ext_R^1(\Omega_{R}M, \Omega_{R}M)=0$ and hence $M$ is free over $S$.
\end{proof}

As a corollary of Theorem \ref{thmARC} we have:

\begin{cor} \label{corARC} Let $R$ be a local ring that has positive depth. If $R$ satisfies $\sacc$, then $R$ satisfies $\arc$.
\end{cor} 

\begin{rmk} Assume $(R, \mathfrak{m})$ is a local ring with $\depth R=0$. Then $R$ satisfies $\sacc$ trivially; $\mathfrak{m}$ is an associated prime of $R$ and hence any module that has constant rank is free. Therefore, when $\depth R=0$, whether $\sacc$ implies $\arc$ or not, is equivalent to the Auslander-Reiten conjecture recorded above.
\end{rmk}

We do not know whether $\sac$ and $\sacc$ are equivalent conditions or not, but $\ac$ and $\sacc$ are not, that is,  $\sacc$ is a weaker condition than $\ac$. We explain this as follows.
(Recall that it follows from Theorem \ref{thmCH} that $\ac$ implies $\sac$.) 

Huneke, \c{S}ega and Vraciu \cite[(4.1)(1)]{HSV} proved that local rings $(R,\mathfrak{m})$ with $\mathfrak{m}^3=0$ satisfy $\sac$.
More precisely, they proved that if $R$ is such a local ring and $M$ is an $R$-module with $\Ext^{i}_{R}(M, M\oplus R)=0$ for $i= n+1, n+2, n+3, n+4$ for some positive integer $n$, then $M$ is free.
On the other hand, Jorgensen and \c{S}ega \cite[(3.3)(2)]{JoS} gave an example of a commutative local ring $(S,\mathfrak{n})$ with $\mathfrak{n}^3=0$ that does not satisfy $\ac$.
Setting $R=S[[X]]$, we see from \cite[(2.3)]{CH2} that $R$ does not satisfy $\ac$.
Moreover, by Theorem \ref{thmARC}(1), $R$ satisfies $\sac$ since $S=R/(X)$ satisfies $\sac$.
These observations show that there exist local rings (of arbitrary depth) that satisfy $\sac$ (and hence $\sacc$) but fail to satisfy $\ac$.

Recently Diveris \cite{Di} defined the following condition for left Noetherian rings $R$, where $\text{ext.deg}(M)=\sup\{n\mid\Ext^{n}_{R}(M,M)\neq 0 \}$:
\begin{quote}
\begin{enumerate}
\item[$\fed$]
The supremum of $\text{ext.deg}(M)$ is finite, where $M$ runs through the $R$-modules with $\text{ext.deg}(M)<\infty$.
\end{enumerate}
\end{quote}
He proved that, when $R$ satisfies $\fed$, the conditions (1) and (2) recorded in the discussion just after Theorem \ref{thmCH} are equivalent, that is, $\fed$ implies $\sac$.
Moreover, he proved that if $R$ satisfies $\fed$, then it satisfies the following {\em generalized Auslander-Reiten condition}, which was also studied in \cite{Wei2,Wei1}:
\begin{quote}
\begin{enumerate}
\item[$\garc$]
For each $R$-module $M$, if there exists a nonnegative integer $n$ such that $\Ext^{>n}_{R}(M,M\oplus R)=0$, then $\pd_{R}M\leq n$.
\end{enumerate}
\end{quote}
We should note that there exists a self-injective noncommutative Artin algebra that fails to satisfy this condition $\garc$ (cf. \cite{Sch}).
We do not know whether all commutative rings satisfy $\garc$ or not.

It is interesting that $\sac$ and $\garc$ are indeed equivalent statements.
Before recording our observation, we introduce a modified version of $\garc$:
\begin{quote}
\begin{enumerate}
\item[$\garcc$]
For each $R$-module $M$, if $\Ext^{\gg 0}_{R}(M,M\oplus R)=0$, then $\pd_{R}M<\infty$.
\end{enumerate}
\end{quote}

\begin{thm}\label{obs}
Let $R$ be a local ring.
The following are equivalent.
\begin{enumerate}[\rm(1)]
\item
$R$ satisfies $\garcc$.
\item
$R$ satisfies $\garc$.
\item
$R$ satisfies $\sac$.
\end{enumerate}
\end{thm}

\begin{proof}
It is trivial that $\garc$ implies $\garcc$.
Let $M$ be an $R$-module with $\pd_RM=n<\infty$.
Then $\Ext^n_{R}(M,R)\ne0$ by \cite[Lemma 1(iii), page 154]{Mat}.
This observation yields that $\garcc$ implies $\garc$.

Assume that $R$ satisfies $\garc$ and that $\Ext_{R}^{>0}(M,R)=\Ext_{R}^{\gg0}(M,M)=0$.
Then there is a nonnegative integer $n$ with $\Ext_{R}^{>n}(M,R)=\Ext_{R}^{>n}(M,M)=0$.
The condition $\garc$ implies $\pd_RM\leq n$.
Since $\Ext_{R}^{>0}(M,R)=0$, we have $n=0$ by \cite[Lemma 1(iii), page 154]{Mat}.
Hence $M$ is free, and we have $\Ext_{R}^{>0}(M,M)=0$.
It follows that $R$ satisfies $\sac$.

Next assume that $R$ satisfies $\sac$ and that $\Ext_{R}^{>n}(M,M\oplus R)=0$ for some $n\ge0$.
Setting $N=\Omega^nM$, we have $\Ext_{R}^{>0}(N,R)=0$.
There are isomorphisms
$$
\Ext_R^i(N,N)\cong\Ext_R^{i-1}(N,\Omega^{n-1}M)\cong\cdots\cong\Ext_R^{i-n}(N,M)\cong\Ext_R^i(M,M)=0
$$
for all $i>n$.
Hence $\Ext_{R}^{>n}(N,N\oplus R)=0$.
Now put $X=N\oplus\Omega N$.
Then for $i>n$ we have:
\begin{align*}
\Ext_R^{i+1}(X,X)
& \cong \Ext_R^{i+1}(N,N)\oplus\Ext_R^{i+1}(N,\Omega N)\oplus\Ext_R^{i+2}(N,N)\oplus\Ext_R^{i+2}(N,\Omega N)\\
& \cong \Ext_R^{i+1}(N,N)\oplus\Ext_R^i(N,N)\oplus\Ext_R^{i+2}(N,N)\oplus\Ext_R^{i+1}(N,N)=0.
\end{align*}
Since $R$ satisfies $\sac$, we have that $\Ext_{R}^{>0}(X,X)=0$, which implies $\Ext_{R}^{1}(N,\Omega N)=0$.
This shows that $N$ is free, and hence $\pd_RM\leq n$.
Thus $R$ satisfies $\garc$.
\end{proof}

Corollary \ref{corARC}, Theorem \ref{obs} and the results discussed above particularly yield the following diagram for commutative local Noetherian rings, where for the implication $\sacc\Longrightarrow\arc$ we assume that the local rings considered have positive depth.
$$
\xymatrix{
& \garcc \ar@{<=>}[d] & \\
& \garc \ar@{<=>}[d] & \\
\fed  \ar@{=>}[r] & \sac \ar@{=>}[r]\ar@{=>}[d]<1.5ex>|= & \sacc \ar@{=>}[r] & \arc \\
& \ac \ar@{=>}[u]
}
$$
We refer the reader to \cite{Di} for details on the relation of the conditions $\fed$ and $\ac$.
It seems worthwhile to pose a question that follows from our previous discussions:

\begin{ques}\label{SACC}
Let $R$ be a local ring that has positive depth.
Assume that $R$ satisfies $\sacc$.
Then does $R$ satisfy $\sac$?
\end{ques}

In view of Theorem \ref{thmARC}, Question \ref{SACC} is equivalent to the following: if $x$ is a nonzerodivisor on a local ring $R$ satisfying $\sacc$, then does $R/xR$ satisfy $\sac$?

\section{Further remarks on the Auslander-Reiten conjecture}

Huneke and Leuschke \cite[1.3]{HL} studied the Auslander-Reiten conjecture and proved the following result:

\begin{thm}[Huneke-Leuschke]\label{Hu-Le}
Let $R$ be a $d$-dimensional complete local Cohen-Macaulay ring which is locally a complete intersection in codimension one.
Assume that $M$ is a maximal Cohen-Macaulay $R$-module of constant rank such that $\Ext_{R}^{i}(M,M)=\Ext_{R}^{i}(M^{\ast},R)=\Ext_{R}^{i}(M,R)=0$ for all $1\le i\le d$.
In the case $d=1$ assume further that $\Ext_{R}^{2}(M,M)=0$ holds\,\footnote[5]{\,This is a missing condition in their original statement. The authors communicated with Huneke and Leuschke for this typo.}.
If $R$ is Gorenstein or if $R$ contains $\QQ$, then $M$ is free.
\end{thm}

Theorem \ref{Hu-Le} provides an important class of Cohen-Macaulay local rings over which the Auslander--Reiten conjecture holds.
Therefore we would like to examine whether the hypothesis that ``$R$ is either Gorenstein or contains $\QQ$" is necessary in Theorem \ref{Hu-Le}.
We do not know whether it is possible or not to remove that assumption, however, following the arguments of the proof of Theorem \ref{Hu-Le}, we deduce that understanding the modules $M$ such that $M^{\ast} \cong M^{\dag}$ may shed light on the Auslander--Reiten conjecture over Cohen-Macaulay local rings that are not necessarily Gorenstein (notice, when $R$ is Gorenstein, $M^{\ast}=\Hom(M,R) \cong \Hom(M,\omega)=M^{\dag}$).
Indeed we have the following concrete statement which shows that the main conclusion of Theorem \ref{Hu-Le} holds under weaker hypotheses for modules $M$ with $M^{\ast} \cong M^{\dag}$.

\begin{prop}\label{newprop}
Let $R$ be a Cohen-Macaulay local ring of dimension $d\ge2$ with a canonical module $\omega$.
Let $M$ be an $R$-module such that:
\begin{enumerate}[\rm(1)]
\item
$M$ is reflexive,
\item
$M$ is locally free on $U_{R}$,
\item
$\Ext_{R}^{i}(M^{\ast},R)=0=\Ext_{R}^{i}(M,R)$ for all $i=1, \dots,d$,
\item
$\Ext_{R}^{d}(M,M)=\Ext_{R}^{d-1}(M,M)=0$.
\end{enumerate}
Then $\depth(M\otimes_{R}M^{\ast}\otimes_{R}\omega)\geq 2$ holds, that is, $M\otimes_{R}M^{\ast}\otimes_{R}\omega$ is reflexive.
In particular, if $M^{\ast} \cong M^{\dag}$, then $M$ is free.
\end{prop}

For the convenience of the reader we record the results which we make use of to prove Proposition \ref{newprop}.

\begin{thm}\cite[1.4]{HL}\label{Thm1}
Let $R$ be a $d$-dimensional Cohen-Macaulay local ring with a canonical module $\omega$, and let $T$ be a maximal Cohen-Macaulay $R$-module.
If $\Ext^{i}_{R}(T,R)=0$ for all $i=1, \dots, d$, then $\omega\otimes_{R}T\cong (T^{\ast})^{\dagger}$ is maximal Cohen-Macaulay.
\end{thm}

\begin{thm}\label{Thm2}\cite[3.2]{CeD}
Let $R$ be a $d$-dimensional Cohen-Macaulay local ring with a canonical module $\omega$, and let $M,N$ be $R$-modules.
Assume that $M$ is locally free on $U_{R}$ and that $N$ is maximal Cohen-Macaulay.
Let $n$ be an integer such that $1\leq n \leq \depth M$.
Then $\depth(M\otimes_{R}N)\geq n$ if and only if $\Ext^i_R(M,N^{\dagger})=0$ for all $i=d-n+1,\dots, d-1,d$.
\end{thm}

\begin{thm}\label{Thm3}{\rm(}Auslander \cite{Au}, see also \cite[3.4]{CeD}{\rm)}
Let $R$ be a Cohen-Macaulay local ring, and let $M$ be a torsionfree $R$-module.
Assume that $M_{\p}$ is free over $R_{\p}$ for each $\p\in \Spec R$ with $\htt\p\leq 1$.
If $M\otimes_{R}M^{\ast}$ is reflexive, then $M$ is free.
\end{thm}

Now we are ready to prove Proposition \ref{newprop}; we thank Graham Leuschke for several discussions on the proof of the proposition.

\begin{proof}[Proof of Proposition \ref{newprop}]
Set $N=M^{\ast}\otimes_{R}\omega$.
It follows from (3) that $M^{\ast}$ is a $d$th syzgy, hence maximal Cohen-Macaulay.
Therefore, setting $T=M^{\ast}$, we have, by Theorem \ref{Thm1}, that $N\cong (M^{\ast\ast})^{\dagger}$ is maximal Cohen-Macaulay. Thus, by (1), $N\cong M^{\dagger}$ so that $N^{\dagger}\cong M^{\dagger\dagger}$. Furthermore, using (1) and (3), we deduce by a similar argument that $M$ is maximal Cohen-Maculay. Therefore $M^{\dagger\dagger}\cong M$ and this yields that $N^{\dagger}\cong M$. Hence, by (4) we have $\Ext^{d}_{R}(M,N^{\dagger})=\Ext^{d-1}_{R}(M,N^{\dagger})=0$.
Theorem \ref{Thm2} shows $\depth(M\otimes_{R}N)\geq 2$.
This and (2) imply that $M\otimes_{R}N=M\otimes_{R}M^{\ast}\otimes_{R}\omega$ is reflexive; see for example \cite[1.4.1(b)]{BH}.

Now assume $M^{\ast} \cong M^{\dag}$. Then $N\cong M^{\dagger}\cong M^{\ast}$ so that $M\otimes_{R}N\cong M\otimes_{R}M^{\ast}$.
Applying Theorem \ref{Thm3}, we conclude that $M$ is free.
\end{proof}

The foregoing discussion shows that the next question is worth recording:

\begin{ques}\label{wm}
Let $R$ be a Cohen-Macaulay local ring with a canonical module.
Assume that there exists a nonfree $R$-module $M$ such that $M^{\ast} \cong M^{\dag}$.
Then is $R$ Gorenstein? 
\end{ques}

We are able to give an affirmative answer to Question \ref{wm} in case the module considered is totally reflexive; we believe such a result is also interesting on its own.
However we do not know more about the existence of such modules in general.

\begin{prop}\label{partialanswer}
Let $R$ be a Cohen-Macaulay local ring with a canonical module $\omega$.
If there exists a totally reflexive $R$-module $M\ne0$ such that $M^{\ast} \cong M^{\dag}$, then $R$ is Gorenstein.
\end{prop}

\begin{proof}
We have isomorphisms
\begin{align*}
M\otimes_R^{\bf L}\omega
& \cong M\otimes_R^{\bf L}\RHom_R(R,\omega)
\cong\RHom_R(\RHom_R(M,R),\omega)\\
& \cong\RHom_R(M^\ast,\omega)
\cong M^{\ast\dag}
\cong M^{\dag\dag}
\cong M.
\end{align*}
Here, the second and third isomorphisms are obtained by \cite[(A.4.24)]{Ch} and the total reflexivity of $M$, respectively.
The fourth isomorphism follows from the fact that $M^\ast$ is maximal Cohen-Macaulay since it is totally reflexive.
These isomorphisms especially says $\Tor^{R}_{>0}(M, \omega)=0$, and hence we have an isomorphism
$$
M\otimes_R\omega\cong M.
$$
Comparing the minimal numbers of generators of both sides, we see that $\omega$ is cyclic, that is, $\omega\cong R$.
Hence $R$ is Gorenstein.
\end{proof}

Next we discuss a question of Huneke and Wiegand \cite{HW} advertised in the introduction:

\begin{ques}[Huneke-Wiegand]\label{queHW}
Is there a nonfree and torsionfree $R$-module $M$ over a one-dimensional Gorenstein domain $R$ such that $M\otimes_{R}M^{\ast}$ is torsionfree?
\end{ques}   

We refer the interested reader to \cite[page 473]{HW} for the motivation of Question \ref{queHW}.
Inspired by such a question (see also Theorem \ref{Thm3}) we define the following condition for commutative Noetherian local rings $R$:
\begin{quote}
\begin{enumerate}
\item[$\hwc$]
For every torsionfree $R$-module $M$, if $M\otimes_{R}M^{\ast}$ is reflexive, then $M$ is free.
\end{enumerate}
\end{quote}

A result of Huneke and Wiegand \cite[(3.7)]{HW} shows that all one-dimensional hypersurface local domains satisfy $\hwc$.
(We do not know whether the same result holds for all complete intersection local domains, cf. \cite{Ce}.)
Examples of local rings (with positive dimension) that do not satisfy $\hwc$ are abundant; for example, if $R=k[[X,Y]]/(XY)$, then $M=M\otimes_{R}M^{\ast}$ is reflexive for the nonfree module $M=R/(X)$.
Notice that the ring $R$ here is not an integral domain.
Furthermore the reflexivity of $M\otimes_{R}M^{\ast}$ is not a very strong assumption for the condition $\hwc$ to hold over local rings:

\begin{prop}\cite[(3.12)]{CeD}\label{p0}
Given an odd integer $d\geq 3$, there exist a $d$-dimensional local hypersurface $R$ with an isolated singularity (hence $R$ is an integral domain) and a nonfree maximal Cohen-Macaulay $R$-module $M$ such that $M\otimes_{R}M^{\ast}$ is torsionfree.
\end{prop}

The following result displays the relation between the condition $\hwc$ and the Auslander-Reiten conjecture for Gorenstein domains (cf. \cite[(3.3), (3.4) and (3.14)]{CeD}):

\begin{prop}\label{HWC-ARC}
Consider the following statements.
\begin{enumerate}[\rm(1)]
\item
All Gorenstein local domains satisfy $\hwc$.
\item
All one-dimensional Gorenstein local domains satisfy $\hwc$.
\item
All Gorenstein local domains satisfy $\arc$.
\end{enumerate}
Then the implications $(1) \Longleftrightarrow (2) \Longrightarrow (3)$ hold.
\end{prop}

We do not know whether $(3)$ implies $(2)$ or not.
Moreover, in view of Proposition \ref{HWC-ARC}, it seems reasonable to conjecture the following.

\begin{conj}
All one-dimensional local domains satisfy $\hwc$.
\end{conj} 

Let $R$ be a local ring, and let $\mathcal{X}(R)$ denote the category of torsionfree $R$-modules $M$ such that $M\otimes_{R}M^{\ast}$ is reflexive.
If $R$ is a complete intersection domain and $M$ is an $R$-module in $\mathcal{X}(R)$ having bounded Betti numbers, then $M$ is free \cite[(4.17)(2)]{Ce}.
We finish this section by recording one more class that satisfy $\hwc$.
This also yields another application of Theorem \ref{Araya} discussed above.

\begin{prop} \label{propHWC} 
Assume $R$ is a reduced ring such that $R=S/(x)$ where $(S, \mathfrak{n})$ is a two-dimensional complete Gorenstein normal local ring and $0\neq x \in \mathfrak{n}$. If $M\in \mathcal{X}(R)$ and $\Tor_{1}^{R}(M,M^{\ast})=0$, then $M$ is free.
\end{prop}

We need two lemmas for the proof of Proposition \ref{propHWC}.
The first one follows from a theorem of Huneke and Jorgensen \cite[5.9]{HJ}:

\begin{lem}\label{727}
Let $R$ be a $d$-dimensional Gorenstein local ring, and let $M,N$ be maximal Cohen-Macaulay $R$-modules.
Assume that $\Ext^{i}_{R}(N,M)$ has finite length for all $i=1, \ldots, d$.
Then $M^{\ast}\otimes_{R}N$ is maximal Cohen-Macaulay if and only if $\Ext^{i}_{R}(N,M)=0$ for all $i=1, \ldots, d$.
\end{lem}

\begin{lem}\label{728}
Let $R$ be a $d$-dimensional Gorenstein local ring, and let $M$ be a maximal Cohen-Macaulay $R$-module.
Let $n$ be a nonnegative integer, and assume that $M$ is locally free on $U_{R}$.
If $\Tor_n^R(M,M^*)=0$, then $\Ext^{n+d}_{R}(M,M)=0$.
\end{lem}

\begin{proof}
Without loss of generality we may assume that $R$ is complete. We will use the following spectral sequence \cite[Theorem 10.62]{Roit}:
$$
E_2^{p,q}=\Ext_{R}^p(\Tor^{R}_q(M,M^*),R)\quad\Longrightarrow\quad H^{p+q}=\Ext_{R}^{p+q}(M,M).
$$
Note that $E_2^{p,q}=0$ if $p>d$ since $R$ has injective dimension $d$.
Furthermore, for $q>0$, $\Tor^{R}_q(M,M^*)$ has finite length.
Therefore $E_2^{p,q}=0$ if $q>0$ and $p\ne d$.
Now the required result follows:
$$
\Ext_{R}^{n+d}(M,M)=H^{n+d}\cong E_2^{d,n}=\Ext_{R}^d(\Tor^{R}_n(M,M^*),R)\cong\Tor^{R}_n(M,M^*)^\vee=0,
$$
where the second isomorphism follows from the local duality theorem.
\end{proof}

\begin{proof}[Proof of Proposition \ref{propHWC}]
Since $R$ is an isolated singularity, $M$ is locally free on $U_{R}$.
By Lemma \ref{728}, the vanishing of $\Tor_1^R(M,M^*)$ forces the vanishing of $\Ext^{2}_{R}(M,M)$. Therefore \cite[(1.7)]{ADS} shows that $M$ lifts to $S$, that is, there exists a finitely generated $S$-module $N$ such that $x$ is $N$-regular and that $N/xN\cong M$. Furthermore, applying \cite[Lemma 2(ii), page 140]{Mat}, we see that $\Ext_{S}^{1}(N,M) \cong \Ext_{R}^{1}(M,M)$. 

By assumption, $M\otimes_RM^\ast$ is reflexive.
As $R$ is $1$-dimensional, $M\otimes_RM^\ast$ is maximal Cohen-Macaulay over $R$.
Lemma \ref{727} implies $\Ext_{R}^{1}(M,M)=0$, and hence $\Ext_{S}^{1}(N,M)=0$.
Note also that the $S$-module $N$ is maximal Cohen-Macaulay.
Since $S$ is an isolated singularity, $N$ is locally free on $U_{S}$.
From the short exact sequence $0\rightarrow N \stackrel{x}{\rightarrow} N \rightarrow M \rightarrow 0$ we obtain the surjection $\Ext_{S}^{1}(N,N) \stackrel{x}{\rightarrow}\Ext_{S}^{1}(N,N)$.
Nakayama's lemma shows $\Ext_{S}^{1}(N,N)=0$.
By Theorem \ref{Araya} the $S$-module $N$ is free, and so is the $R$-module $M$.
\end{proof}

An example we discussed before shows that the vanishing of $\Tor_{1}^{R}(M,M^{\ast})$ in Proposition \ref{propHWC} is necessary: for $R=k[[X,Y]]/(XY)$ and $M=R/(X)$, setting $S=k[X,Y]$, we have that $R=S/(XY)$, $M\in \mathcal{X}(R)$, $\Tor_{1}^{R}(M,M^{\ast})\cong k\ne0$ and $M$ is nonfree.

\section*{Acknowledgments}
The authors would like to thank Tokuji Araya, Luchezar Avramov, Lars Christensen, Hailong Dao, Kosmas Diveris, Craig Huneke, Graham Leuschke, Greg Piepmeyer and Roger Wiegand for their valuable comments during the preparation of this paper.
This work was done during the visit of the first author to University of Nebraska-Lincoln in May and July, 2011.
He is grateful for their kind hospitality.
The authors also thank the referee for his/her careful reading and useful suggestions.

\end{document}